\documentclass{amsart}

\usepackage{amsmath}
\usepackage{amssymb}
\usepackage{mathrsfs}
\usepackage{amsthm}
\usepackage{comment}

\title{On $\H_Y$-ideals}

\author[M. Badie]{Mehdi Badie}
\address{Jundi-Shapur University of Technology, Dezful, Iran}
\email{badie@jsu.ac.ir}



\theoremstyle{plain}
\newtheorem{Thm}{Theorem}[section]

\newtheorem{Def}[Thm]{Definition}
\newtheorem{Pro}[Thm]{Proposition}
\newtheorem{Cor}[Thm]{Corollary}
\theoremstyle{definition}

\newcommand{\cal}{\mathcal}

\newcommand{\N}{\mathbb{N}}

\newcommand{\Ge}[1]{\big< #1 \big>}
\renewcommand{\H}{\mathcal{H}}
\newcommand{\ff}{if and only if }
\newcommand{\Mi}{\mathrm{Min}}
\newcommand{\Ma}{\mathrm{Max}}
\newcommand{\An}{\mathrm{Ann}}
\newcommand{\Sp}{\mathrm{Spec}}
\newcommand{\Rad}{\mathrm{Rad}}

\newcommand{\maxl}{\mathrm{maxl}}

\begin{document}

\begin{abstract}
	In this article, we continue the studying of $ \H_Y $-ideals. We introducing two notions fixed and free $ H_Y $-ideals as an extension of fixed and free $ z $-ideals in $ C(X) $ and relative $ \H_Y $-ideals as an extension of relative $ z $-ideals. It has been shown that a large amount of the results of the  mentioned papers and generally the papers in the literature about these topics, are special cases of the results of this paper. We prove that $ Y $ is compact \ff every proper $ \H_Y $-ideal is a fixed $ \H_Y $-ideal; \ff every proper strong $ \H_Y $-ideal is a fixed ideal. Also,  we show that every proper ideal is a relative $ \H_Y $-ideal, \ff every proper ideal is a strong $ \H_Y $-ideal; \ff $ R $ is regular.	
\end{abstract}

\maketitle

\section{Introduction}

The concept of $z$-ideal, was first studied in the rings of continuous functions as an ideal $I$ of $C(X)$ that $Z(f)\subseteq Z(g)$ and $f\in I$ implies that $g\in I$, see \cite{gillman1960rings}. Then this concept was studied more generally for the commutative rings, in \cite{mason1973z}, as an ideal $I$ of $R$ that whenever two elements of $R$ are contained in the same family of maximal ideals and $I$ contains one of them, then it follows that $I$ contains the other one. If we  use $(Z(f))^{\circ}\subseteq (Z(g))^{\circ}$ instead of the above inclusion relation and the minimal prime ideals instead of the maximal ideals in the above definitions,  then we obtain the concept of $z^{\circ}$-ideal ($d$-ideal) in $C(X)$ and the commutative rings, which are introduced and carefully studied in  \cite{azarpanah1999z,azarpanah2000ideals,huijsmans1980z}. The concepts of $z$-ideal and $z^{\circ}$-ideal can be generalized to the concepts of $sz$-ideal and $sz^{\circ}$-ideal ($\xi$-ideal), respectively, based on the finite subsets of the ideals instead of the single points in the ideal, and are studied in \cite{aliabad2011sz,Artico1981,mason1973z}.

The concepts of $\H_Y$-ideals and the strong $\H_Y$-ideals are generalizations of $ z $-ideals, strong $ z $-ideals, $ z^\circ $-ideals and strong $ z^\circ $-ideals. These concepts have been introduced and carefully studied in \cite{badie2019extension}. It has been shown that a large amount of the results of the above mentioned papers and generally the papers in the literature about these topics, are special cases of the results of \cite{badie2019extension}. In this paper, we continue the studying of these concepts.  

In the next section we recall some pertinent definitions. In Section 3, we introduce and study the fixed and free $ \H_Y $-ideals and free and fixed $ \H_Y $-ideals respect to a subset of $ Y $. In this section we show that $ Y $ is compact \ff every proper $ \H_Y $-ideal is a fixed $ \H_Y $-ideal; \ff every proper strong $ \H_Y $-ideal is a fixed ideal. Also, we give a compactification of $ Y $. The Section 4, devoted to introducing and studying of relative $ \H_Y $-ideal as an extension of relative $ z $-ideals. In this section, we show that every proper ideal is a relative $ \H_Y $-ideal, \ff every proper ideal is a strong $ \H_Y $-ideal; \ff $ R $ is regular.  

\section{Prelementary}

In this article, any ring $ R $ is commutative with unity.  A semi-prime ideal is an ideal which is an intersection of prime ideals. For each ideal $ I $ of $ R $ and each element $ a $ of $ R $, we denote the ideal $ \{ x \in R : ax \in I\}$ by $ ( I : a ) $. When $ I = \{ 0 \}  $ we write $ \An(a) $ instead of $(\{0\}:a)$ and call this the annihilator of $a$. If $ \An(a) $ is maximal in the set of all annihilator of nonzero elements of $R$, then $ \An(a)$ is a prime ideal, and it is called an affiliated prime ideal. A prime ideal $ P $ containing an ideal $I$ is said to be a minimal prime  over $ I $, if there are not prime ideals strictly contained in $ P $ that contains $ I $. $\Sp(R)$, $\Mi(R)$, and $\Rad(R)$ denote the set of all prime ideals, all minimal prime 	ideals, all maximal ideals of $R$  and their intersections, respectively. By $\Mi(I)$ we mean the set of minimal prime ideals of $I$. In fact $\Mi((0))=\Mi(R)$. A ring $R$ is said to be reduced if  $\Rad(R)=(0)$.

A prime ideal $ P $ is called a Bourbaki associated prime divisor of an ideal $ I $ if $ ( I‌ : x ) = P $, for some $ x \in R $. We denote the set of all Bourbaki associated prime divisors of an ideal $ I $ by $ \cal{B}(I) $. We use $ \cal{B}(R) $ instead of $ \cal{B}(\{0\}) $. A representation $ I = \bigcap_{P‌ \in \cal{P}} P $ of $ I $ as an intersection of prime ideals is called irredundant if no $ P \in \cal{P} $ may be omitted. Let $ I $ be a semi-prime ideal, $ P_\circ \in \Mi(I) $ is called irredundant with respect to $ I $,  if $ I \neq \bigcap_{ P_\circ \neq P \in \Mi(I) } P$. If $ I $ is equal to the intersection of all irredundant ideals with respect to $ I $, then we call $ I $ a fixed-place ideal, exactly, by \cite[Theorem 2.1]{aliabad2013fixed}, we have $ I = \bigcap \cal{B}(I) $.

In this paper, all $Y\subseteq \Sp(R)$ is considered by Zariski topology; i.e.,  by assuming as a base for the closed sets of $Y$, the sets $h_Y(a)$ where $h_Y(a)=\{ P\in Y:a\in P\}$. Hence, closed sets of $Y$ are of the form $h_Y(I)=\bigcap_{a\in I}h_Y(a)=\{P\in Y: I\subseteq P\}$, for some ideal $I$ in $R$. Also, we set $h_Y^c(I)=Y\backslash h_Y(I)$. For any subset $S$ of $Y$, we show the kernel of $S$ by $k(S)=\bigcap_{P\in S}P$ and  we have $\overline{S}=cl_Y S=h_Yk(S)$. When $Y=\Sp(R)$, we omit the index $Y$ and when $Y=\Ma(R)$ ($Y=\Mi(R)$) we write  $M$ (m) instead of $Y$ in the index. By these notations, for every $S\subseteq R$, we can use the notations $kh_m(S)$ and $kh_M(S)$ instead of $P_S$ and $M_S$ (which is usually used in the context of $C(X)$), respectively.

The reader is referred to \cite{atiyah1969introduction}, \cite{gillman1960rings}, \cite{lam1991first} and \cite{stephen1970general} for undefined terms and notations. 

Let $R$ be a ring, $Y \subseteq \Sp(R)$ and $I$ be an ideal of $R$. Then, by \cite[Proposition3.2]{badie2019extension} the following are equivalent:

\begin{itemize}
	\item[(a)] For every $a\in I$ and  $S\subseteq R$, it follows from $h_Y(a)\subseteq h_Y(S)$ that $S\subseteq I$.
	\item[(b)] For every $a\in I$ and  $S\subseteq R$, it follows from $h_Y(a)=h_Y(S)$ that $S\subseteq I$.
	\item[(c)] For every $a\in I$ and  $b\in R$, it follows from $ h_Y(a)=h_Y(b)$ that $b\in I$.
	\item[(d)] For every $a\in I$ and  $b\in R$, it follows from $h_Y(a)\subseteq h_Y(b)$ that $b\in I$.
	\item[(e)]  If $a\in I$, then $kh_Y(a)\subseteq I$.
	\item[(f)] For every $a\in I$ and  $S\subseteq R$, it follows from $kh_Y(S)\subseteq kh_Y(a)$ that $S\subseteq I$.
	\item[(g)] For every $a\in I$ and  $S\subseteq R$, it follows from $kh_Y(S)=kh_Y(a)$ that $S\subseteq I$.
	\item[(h)] For every $a\in I$ and  $b\in R$, it follows from $ kh_Y(b)=kh_Y(a)$ that $b\in I$.
	\item[(k)] For every $a\in I$ and  $b\in R$, it follows from $kh_Y(b)\subseteq kh_Y(a)$ that $b\in I$.
\end{itemize}

An ideal $I$ of $R$ is said to be an $\H_Y$-ideal if it satisfies in the above equivalent conditions.

Let $R$ be a ring, $Y \subseteq \Sp(R)$ and $I$ be an ideal of $R$. Then, by \cite[Proposition 3.4]{badie2019extension}, the following are equivalent:

\begin{itemize}
	\item[(a)] For every finite subset $F$‌ of $I$ and every $S\subseteq R$, it follows from  $h_Y(F)=h_Y(S)$ that $S\subseteq I$.
	\item[(b)] For every finite subset $F$‌ of $I$ and every finite subset $G$ of $R$, it follows from  $h_Y(F)=h_Y(G)$ that $G\subseteq I$.
	\item[(c)] For every finite subset $F$‌ of $I$ and every finite subset $G$  of $R$, it follows from  $h_Y(F)\subseteq h_Y(G)$ that $G\subseteq I$.
	\item[(d)] It follows from  $h_Y(a)\in \H_Y(I)$  that $a\in I$.
	\item[(e)] For every finite subset $F$ ‌of $R$, it follows from  $h_Y (F)\in \H_Y (I)$ that $F\subseteq I$.
	\item[(f)] For every finite subset $F$ of $I$ and $a\in R$‌, it follows from  $h_Y(F)=h_Y(a)$ that $a\in I$.
	\item[(g)] For every finite subset $F$ of $I$ and $a\in R$, it follows from  $h_Y(F)\subseteq h_Y(a)$ that $a\in I$.
	\item[(k)] For every finite subset $F\subseteq I$, we have $ kh_Y(F)\subseteq I$.
	\item[(l)] For every finite subset $F$ of $I$ and $a\in R$‌, it follows from  $kh_Y(a)=kh_Y(F)$ that $a\in I$.
	\item[(m)] For every finite subset $F$ of $I$ and $a\in R$, it follows from  $kh_Y(a)\subseteq kh_Y(F)$ that $a\in I$.
	\item[(n)] For every finite subset $F$ of $I$ and any $S\subseteq R$‌, it follows from  $kh_Y(S)=kh_Y(F)$ that $S\subseteq I$.
	\item[(o)] For every finite subset $F$ of $I$ and any $S\subseteq R$, it follows from  $kh_Y(S)\subseteq kh_Y(F)$ that $S\subseteq I$.
\end{itemize}

An ideal $I$ of $R$ is said to be a strong $\H_Y$-ideal if it satisfies in the above equivalent conditions. An ideal $I$ of $R$ is called a $Y$-Hilbert ideal, if $I$ is  an intersection of elements of some subfamily of $Y$; i.e., $I=kh_Y(I)$.  By $\maxl(E)$, we mean the set of all maximal elements of $E$. Also, we denote by $S\H_Y({\cal{A}})$ ($PS\H_Y({\cal{A}})$) the set of all strong $H_Y$-ideals (proper strong $H_Y$-ideals) of $A$. See \cite{badie2019extension}, for more information about the $ \H_Y $-ideals and strong $ \H_Y $-ideals.

\section{$\H_Y$-fixed ideals and $\H_Y$-free ideals}

In this section, we introduce and study the fixed and free $ \H_Y $-ideals and free and fixed $ \H_Y $-ideals respect to a subset of $ Y $ as extensions of fixed and free ideals in $ C(X) $. Also, we give some equivalent properties to compaction of $ Y $

\begin{Def}
	Let $Y \subseteq \Sp(R)$ and $I$ be an ideal of $R$. Then $I$ is called an $\H_Y$-fixed ideal, if $\H_Y(I)$ is a fixed $\H_Y$-filter; i.e. $\bigcap \H_Y(I) \neq \emptyset$. If $I$ is not an $\H_Y$-fixed ideal, then $I$ is called $\H_Y$-free ideal.
\end{Def}

\begin{Pro}
	Let $Y \subseteq \Sp(R)$ and $I$ be an ideal of $R$. Then the following statements hold.
	\begin{itemize}
		\item[(a)] $\bigcap \H_Y(I) =h_Y(I)$.
		\item[(b)] $I$ is an $\H_Y$-fixed ideal \ff $h_Y(I)\neq\emptyset$.
	\end{itemize}
	\label{intersection of fixed}
\end{Pro}
\begin{proof} (a). Let $F_I$ be the set of all finite subsets of $I$, then we can write 
	
	\[ P \in h_Y(I) \quad \Leftrightarrow \quad I\subseteq P \quad \Leftrightarrow \quad \forall S\in F_I, \quad S\subseteq P \]
	\[ \Leftrightarrow \quad P\in\bigcap_{S\in F_I}h_Y(S)=\bigcap \H_Y(I)\quad \Leftrightarrow \quad P\in\bigcap \H_Y(I).\]
	
	(b). By (a), it is clear.
\end{proof}

Now we can conclude the following corollary from the above proposition.

\begin{Cor}
	Let $Y \subseteq \Sp(R)$ and $I$ be an ideal of $R$. The following hold.
	\begin{itemize}
		\item[(a)] If $I$ is a $Y$-Hilbert ideal, then $I$ is a fixed $\H_Y$-ideal.
		\item[(b)]  $I$ is $\H_Y$-fixed (resp., $\H_Y$-free) \ff $\H_Y^{-1} \H_Y(I)$ is an $\H_Y$-fixed ideal ($\H_Y$-free ideal).
		\item[(c)] $ \H_Y^{-1} \H_Y(I) = R $ \ff $ I \not\subseteq \bigcup Y $.
	\end{itemize}
	\label{H_Y Hilbert is fixed}
\end{Cor}

It is clear that every $\H_Y$-fixed ideal is proper. Moreover the following  corollary is immediate.

\begin{Cor}
	Let $Y \subseteq \Sp(R)$ and $I$ be an ideal of $R$. Then $I$ is a maximal $\H_Y$-fixed ideal \ff $I\in \maxl(Y)$.
	\label{maximal fixed}
\end{Cor} 

Note that a maximal fixed $\H_Y$-filter is a also an $\H_Y$-ultrafilter. Therefore, $I$ is a maximal $\H_Y$-fixed ideal \ff $\H_Y(I)$ is a fixed $\H_Y$-ultrafilter. Hence, clearly, an $\H_Y$-filter $\mathscr{U}$ is a fixed  $\H_Y$-ultrafilter \ff $\H_Y^{-1}(\mathscr{U})$ is a  maximal element of $Y$.

\begin{Thm}
	Let $ Y \subseteq \Sp(R) $. The following statements are equivalent.
	\begin{itemize} 
		\item[(a)] Every proper $\H_Y$-ideal is a fixed $\H_Y$-ideal.
		\item[(b)] Every proper strong $\H_Y$-ideal is a fixed $\H_Y$-ideal.
		\item[(c)] Every ideal contained in $ \bigcup Y $ is an $\H_Y$-fixed ideal.
		\item[(d)] $ \maxl(PS\H_Y) \subseteq Y $.
		\item[(e)] $ Y $ is compact.
	\end{itemize}		
	\label{compact and fixed H-Y}
\end{Thm}
\begin{proof}
	(a)$ \Rightarrow $(b). It is clear.
	
	(b)$ \Rightarrow $(c). It follows from Corollary \ref{H_Y Hilbert is fixed}.
	
	(c)$ \Rightarrow $(d). Suppose that $M \in \maxl(PS\H_Y)$, then  $M$ is an $\H_Y$-fixed ideal, so $h_Y(M)\ne \emptyset$, by Proposition \ref{intersection of fixed}, so $M\in Y$, consequently $\Ma(R)\subseteq Y$.
	
	(d)$ \Rightarrow $(e). Suppose that $ \{ h_Y(a_\alpha) \}_{\alpha \in A} $ is a family of closed basic with the finite intersection. Then there is a proper $\H_Y$-filter $\mathscr{F}$ containing $ \{ h_Y(a_\alpha) \}_{\alpha \in A} $. Then an ultrafilter $ \mathscr{U} $ on $ Y $ exists such that $ \mathscr{F} \subseteq \mathscr{U} $. Since $ \H_Y^{-1} (\mathscr{U}) \in \maxl(PS\H_Y) $, $ \H_Y^{-1} (\mathscr{U}) $ is $ \H_Y $-fixed and therefore $ \mathscr{U} $ is fixed, thus $ \mathscr{F} $ is fixed, hence $\emptyset \neq \bigcap \mathscr{F} \subseteq \bigcap_{\alpha \in A} h_Y(a_\alpha) $. Consequently $Y$ is compact.
	
	(e)$ \Rightarrow $(a). Since $ \H_Y(I) = \{ h_Y(F) : F \text{ is a finite subset } I \}$ is a family closed sets with the finite intersection property,  It is clear.
\end{proof}

\begin{Cor}
	Suppose that $ Y \subseteq \Mi(R) $ and $ \bigcap Y = \{ 0 \} $. $ Y $ is compact \ff $ \mathcal{B}(R) \subseteq Y $.
\end{Cor}
\begin{proof}
	It concludes from \cite{aliabad2013fixed} and Theorem \ref{compact and fixed H-Y}. 
\end{proof}

\begin{Thm}
	Suppose that $ Y \subseteq \Sp(R) $. If $ Z = \maxl(PS\H_Y) \cup Y $, then $ Z $ is a compactification of $ Y $.
\end{Thm}
\begin{proof}
	$ Y $ is dense in $ Z $ and $ \maxl(PS\H_Z) \subseteq \maxl(PS\H_Y) \subseteq Z $, so $ Z $ is a compactification of $ Y $.
\end{proof}

Suppose that $X$ is a topological space and $Y = \Ma(C(X))$. For every $p \in \beta X \setminus X $, the maximal ideal $M^p$ is a free $z$-ideal which is an $\H_Y$-fixed ideal. This example shows that the notion of fixed ideals and notion of $ \H_Y $-fixed ideals need not coincide in the rings of continuous functions literature. In the following definition we introduce new notion which coincides with the fixed ideals notion in the rings of continuous functions literature.

\begin{Def}
	Suppose $Y \subseteq \Sp(R)$ and $S \subseteq Y$. An ideal $I$ is called an $\H_Y$-fixed ideal with respect to $S$ if $ \big( \bigcap \H_Y(I) \big) \cap S \neq \emptyset $ and $I$ is called an $\H_Y$-free ideal with respect to $S$ if $I$ is not an $\H_Y$-fixed ideal with respect to $S$; i.e. $ \big( \bigcap \H_Y(I) \big) \cap S = \emptyset $
\end{Def}

Suppose that $S\subseteq Y\subseteq \Sp(R)$. Clearly if an ideal $I$ is $\H_Y$-fixed with respect to $S$, then it is also an $\H_Y$-fixed; but the converse is not true in general. If $Y = \Ma(C(X))$ and $S = \{ M_p : p \in X \}$, then the family of all $\H_Y$-fixed ideals (in fact, fixed $z$-ideals) and the family of all $\H_Y$-fixed ideal with respect to $S$ coincide.

\begin{Thm}
	Let $S \subseteq Y \subseteq \Sp(R)$ and $I$ be an ideal of $R$. $I$ is an $\H_Y$-fixed ideal with respect to $S$ \ff $I$ is a $\H_S$-fixed ideal.
	\label{fixed H-Y and S-fixed}
\end{Thm}
\begin{proof}
	We can conclude from Proposition \ref{intersection of fixed}, that
	\[ \big( \bigcap \H_Y(I) \big) \cap S = h_Y(I) \cap S = h_S(I) = \bigcap \H_S(I) \]
	This completes the proof.  
\end{proof}

Now Corollary \ref{maximal fixed}, and Theorem \ref{fixed H-Y and S-fixed}, conclude the following corollary.

\begin{Cor}
	If $S \subseteq Y \subseteq \Sp(R)$, then $\maxl(Y) \cap S$ and the family of maximal $\H_Y$-fixed ideal with respect to $S$ coincide.
\end{Cor} 

\begin{Pro}
	Suppose $R'$ is a subring of a ring $R$ and $Y \subseteq \Sp(R)$ , then $Y' = \{ P \cap R' : P \in Y \} \subseteq \Sp(R') $. Set \[ \mathscr{F}' = \{ h_{Y'}(F) : h_Y(F) \in \mathscr{F} \text{ and } F \text{ is a finite subset of } R' \}, \]
	for every $\H_Y$-filter $\mathscr{F}$.  If $I$ is an $\H_Y$-fixed ideal, then $I\cap R'$ is an $\H_{Y'}$-fixed ideal.
\end{Pro}
\begin{proof}
	The proof is straightforward.
\end{proof}

\section{Relative $\H_Y$-ideals}

In this section we introduce and study relative $ \H_Y $-ideals as an extension of relative $ z $-ideals and we show that the most important results about the relative $ z $-ideals  can be extended to relative $ \H_Y $-ideals.

\begin{Def}
	Let $Y \subseteq \Sp(R)$ and $I$ and $J$ be ideals of $R$. $I$ is called (resp., strong) $\H_{YJ}$-ideal if for every $a \in I$ (resp., finite subset $F$ of $I$) we have $kh_Y(a) \cap J \subseteq I$ (resp., $kh_Y(F) \cap J \subseteq I$). If $I$ is (resp., strong) $H_{YJ}$ ideal, for some ideal $J \not\subseteq I$, then $I$ is called relative (resp., strong) $\H_Y$-ideal and $ J $ is called a (resp., strong) $ \H_Y $-factor of $ I $. A (resp., strong) $ \H_{YJ} $-ideal $ I $ is called a trivial (resp., strong) $ \H_{YJ} $-ideal, if $ J \subseteq I $, otherwise $ I $ is called a nontrivial $ \H_{YJ} $-ideal.
\end{Def}

\begin{Pro}
	If $Y \subseteq \Sp(R)$, then the following hold.
	\begin{itemize}
		\item[(a)] If $J$ is an ideal of $R$, then every strong $\H_{YJ}$-ideal is an $\H_{YJ}$-ideal.
		\item[(b)] Every strong relative $\H_Y$-ideal is a relative $\H_Y$-ideal.
		\item[(c)] Every ideal $I$ is a strong $\H_{YI}$-ideal.
		\item[(d)] Every (resp., strong) $\H_Y$-ideal is a (resp., strong) $\H_{YJ}$-ideal, for every ideal $J$ of $R$.
		\item[(e)] Every (resp., strong) $\H_Y$-ideal is a relative (resp., strong) $\H_Y$-ideal.
		\item[(f)] Suppose that $J$ is a (resp., strong) $\H_Y$-ideal. Every (resp., strong) $\H_{YJ}$-subideal of $J$ is a (resp., strong) $\H_Y$-ideal.  
	\end{itemize}
\label{Primary relative}
\end{Pro}
\begin{proof}
	It is straightforward.
\end{proof}

\begin{Thm}
	Let $Y \subseteq \Sp(R)$ and $I$ and $J$ be ideals of $R$. Then the following hold
	\begin{itemize}
		\item[(a)] If $I$ is a (resp., strong) $\H_{YJ}$-ideal and $P \in \Mi(R)$, then $P$ is a (resp., strong) $\H_{YJ}$-ideal.
		\item[(b)] A prime ideal $ P $  is a (resp., strong) $\H_{YJ}$-ideal \ff $ P $ is either a (resp., strong) $ \H_Y $-ideal or $ J \subseteq P $.
		\item[(c)] If $ I $ is a relative (resp., strong) $ H_Y $-ideal, then $ \Mi(I) $ has a (resp., strong) $ \H_Y $-ideal.
	\end{itemize}
\label{Prime and H-ideal}
\end{Thm}
\begin{proof}
	(a). Suppose that $I$ is an $\H_{YJ}$-ideal. If $a \in P$, then $b \notin P$ exists such that $(ab)^n \in I$, for some $n \in \N$. Since $I$ is a $\H_{YJ}$-ideal, it follows that 
	\[ kk_Y(a) \cap kh_Y(b) \cap J = k \big( h_Y(a) \cup h_Y(b) \big) \cap J = k h_Y(ab) \cap J = kh_Y(ab)^n \cap J \subseteq I \subseteq P \]
	Since $kh_Y(b) \not\subseteq P$ and $P$ is prime, it follows that $ kh_Y(a) \cap J \subseteq P $. Consequently, $P$ is an $\H_{YJ}$-ideal.
	
	Now Suppose that $I$ is a strong $\H_Y$-ideal. If $F = \{ a_1 , a_2 , \ldots , a_n \}  \subseteq P$, then $ b_i \notin P $ and $m_i \in \N$ exist such that $ (a_i b_i)^{m_i} \in I$, for every $ 1 \leqslant i \leqslant n $. Set $ b = b_1 b_2 \ldots b_n \notin P $ and $ m = \max\{ m_1 , m_2 , \ldots , m_n \}$, then  $ (a_i b)^m  \in I $, for each $ 1 \leqslant i \leqslant n $. 
	Supposing $ E = \{ a_i^m b^m : 1 \leqslant i \leqslant n \} $, $ kh_Y(E) \cap J \subseteq I$, because $I$ is a strong $\H_{YJ}$-ideal. Thus
	\[ kh_Y(b) \cap kh_Y(F) \cap J = k \big( h_Y(b) \cup h_Y(F) \big) \cap J = k(h_Y(E)) \cap J \subseteq I \subseteq P \] 
	Since $kh_Y(b) \not\subseteq P$ and $P$ is prime, it follows that $ kh_Y(F) \cap J \subseteq P $. Consequently, $P$ is a strong $\H_{YJ}$-ideal.
	
	(b). It is clear if $ J \subseteq P $, then $ P $ is (resp., strong) $ \H_{YJ} $-ideal. Now suppose that $ J \not\subseteq P $ and $a \in P$ (resp., $F$ is a finite subset of $ P $), then $ k h_Y(a) \cap J \subseteq P $ (resp., $ k h_Y(F) \cap J \subseteq P $) implies that $ k h_Y(a) \subseteq P $ ($ k h_Y(F) \subseteq P $). Thus $ P $ is (resp., strong) $ \H_{YJ} $-ideal.
	
	(c). Since $ I $ is a relative (resp., strong) $ \H_Y $-ideal, there is some ideal $ J \supseteq I $ such that $ I $ is a (resp., strong) $ \H_{YJ} $-ideal. Thus $ a \in J \setminus I $ exists, Set $ K = \Ge{a} $, clearly, $ I $ is a strong $ \H_{YK} $-ideal. If $ K \subseteq \sqrt{I} $, then $ a^n \in I $, for some $ n \in \N $. Thus 
	\[ kh_Y(a) = kh_Y(a) \cap K = kh_Y(a^n) \cap K \subseteq I  \]
	This shows that $ a \in I $, which is a contradiction. Hence $ K \not\subseteq \sqrt{I} $, so $ K \not\subseteq P $, for some $ P \in \Mi(R) $. Now (a) and (b) follow that $ P $ is a (resp., strong) $ \H_Y $-ideal.
\end{proof}

\begin{Thm}
	Let $Y \subseteq \Sp(R)$ and $I$ and $J$ be ideals of $R$. The following statements are equivalent.
	\begin{itemize}
		\item[(a)] $I$ is a (resp., strong) $\H_{YJ}$-ideal.
		\item[(b)] $I_\H \cap J \subseteq I$ ($I_{S\H} \cap J \subseteq I$).
		\item[(c)] $I_\H \cap J = I \cap J$ ($I_{S\H} \cap J = I \cap J$).
		\item[(d)] A (resp., strong) $\H_Y$-ideal $K$ containing $I$ exists such that $K \cap J = I \cap J$.
		\item[(e)] A (resp., strong) $\H_Y$-ideal $K$ containing $I$ exists such that $K \cap J \subseteq I$.
		\item[(f)] For each $a \in I$ and $b \in J$ (resp., finite subset $F$ of $I$ and finite subset $E$ of $J$), $h_Y(a) \subseteq h_Y(b) $ (resp., $h_Y(F) \subseteq h_Y(E)$)  implies that $b \in I$ (resp., $E \subseteq I$).
	\end{itemize}
\label{Equivalnet condition by relative ideal}
\end{Thm}
\begin{proof}
	We prove the theorem for $\H_Y$-ideals, similarly one can prove it for strong $\H_Y$-ideals.
	
	(a) $ \Rightarrow $ (b). By Theorem \ref{Prime and H-ideal}, either  $ P_\H = P $ or $ J \subseteq P $, for each $P \in \Mi(I)$, so \cite[Proposition 7.11]{badie2019extension}, follows that
	\begin{align*}
	I_\H \cap J & = ( \sqrt{I} )_\H \cap J = \big( \bigcap_{ P \in \Mi(I) } P  \big)_\H \cap J  \\
				& \subseteq \big( \bigcap_{ P \in \Mi(I) } P_\H \big) \cap J = \big( \bigcap_{ P \in \Mi(I) } P \big) \cap J = \sqrt{I} \cap J
	\end{align*}
	If $ a \in I_\H \cap J $, then $ a \in \sqrt{I} \cap J $, so $  \in J $ and $ a^n \in J $, for some $ n \in \N $. Hence 
	$ a \in kh_Y (a) \cap J = kh_Y (a^n) \cap J \subseteq I $ and therefore $ I_\H \cap J \subseteq I $.
	
	(b) $ \Rightarrow $ (c). It is clear.
	
	(c) $ \Rightarrow $ (d). Set $ K = I_\H $.
	
	(e) $ \Rightarrow $ (f). It is clear.
	
	(e) $ \Rightarrow $ (f). $ h_Y(a) \subseteq h_Y(b) $ concludes that $ kh_Y(b) \subseteq kh_Y(a) $, since $ a \in I \subseteq K $, it follows that $ b \in kh_Y(b) \cap J \subseteq kh_Y(a) \cap J \subseteq K \cap J \subseteq I $.
	
	(f) $ \Rightarrow $ (a). Suppose $ a \in I \cap J $. If $ b \in kh_Y(a) $, then $ h_Y(a) \subseteq h_Y(b) $, so $ b \in I $, by the assumption. Hence $ kh_Y(a) \subseteq I $. It follows that $I$ is an $ \H_{YJ} $-ideal.  
\end{proof}

\begin{Pro}
	Let $ Y \subseteq \Sp(R) $. The following statements hold.
	\begin{itemize}
		\item[(a)] Suppose that $ \{I_\alpha \}_{ \alpha \in A} $ and $ \{ J_\alpha \}_{\alpha \in A} $ are two families of ideals of $R$. If $ I = \bigcap_{ \alpha \in A } I_\alpha $, $ J = \bigcap_{\alpha \in A} J_\alpha $ and $ I_\alpha $ is a (resp., strong) $ \H_{YJ} $-ideal, then $ I $ is a (resp., strong) $\H_{YJ}$-ideal.
		\item[(b)] If $ J \subseteq K $ are two ideals of $ R $ and $ I $ is a (resp., strong) $ \H_{YK} $-ideal, then $ I $ is a (resp., strong) $ \H_{YJ} $-ideal.
		\item[(c)] If $ \{ I_\alpha \}_{\alpha \in A} $ is a family of (resp., strong) $ \H_{YJ} $-ideals, then $ \bigcap_{\alpha \in A} I_\alpha $ is a (resp., strong) $ \H_{YJ} $-ideal.
		\item[(d)] If $ I $ is a (resp., strong) $ \H_{YJ} $-ideal and $ J $ is a (resp., strong) $ \H_{YK} $-ideal, then $ I $ is a (resp., strong) $ \H_{YK} $-ideal.
		\item[(e)] If $ I $ is a (resp., strong) $ \H_{YJ} $-ideal, then $ \sqrt{I} $ is a (resp., strong) $ \H_{Y\sqrt{J}} $-ideal.
		\item[(f)] $ I $ is a (resp., strong) $ \H_{YJ} $-ideal \ff $ I \cap J $ is a $ \H_{YJ} $-ideal.
		\item[(g)] $ I $ is a (resp., strong) $ \H_{YJ} $-ideal \ff $ I $ is a (resp., strong) $ \H_{Y(I+J)} $-ideal.
		\item[(h)] Suppose that $ J $ is a (resp., strong) $ \H_Y $-ideal containing an ideal $ I $. $ I $ is a (resp., strong) $\H_{YJ} $-ideal \ff $ I $ is a (resp., strong) $ \H_Y $-ideal.
		\item[(i)] Suppose that $ J $ is a (resp., strong) $ \H_Y $-ideal. $ I $ is a (resp., strong) $ \H_{YJ} $-ideal \ff $ I \cap J $ is a (resp., strong) $ \H_Y $-ideal.
		\item[(j)] Suppose that $ I $ and $ J $ are ideals of $ R $. $ I \cap J $ is both (resp., strong) $ \H_{YI} $-ideal and $ \H_{YJ} $-ideal \ff $ I $ is a (resp., strong) $ \H_{YJ} $-ideal and $ J $ is a (resp., strong) $ \H_{YI} $-ideal.
		\item[(k)] Suppose that $ I $ and $ J $ are ideals of $ R $. $ I_\H \cap J $ (resp., $ I_{S\H} \cap J $) is the smallest (resp., strong) $ \H_{YJ} $-ideal containing $ I \cap J $.
		\item[(l)] Suppose that $ I \subseteq K $ and $ J $ are ideals of $ R $ and $ I_\H = K_\H $ (resp., $ I_{S\H} = K_{S\H} $). If $ I $ is a (resp., strong) $ \H_{YJ} $-ideal, then $ K $ is a (resp., strong) $ \H_{YJ} $-ideal.
		\item[(m)] Suppose that $ I $, $ J $ and $ K $ are ideals of $ R $. If $ I \subseteq K \subseteq I_\H $ (resp., $ I \subseteq K \subseteq K \subseteq I_{S\H} $) and $ I $ is a (resp., strong) $ \H_{YJ} $-ideal, then $ K $ is a (resp., strong) $ \H_{YJ} $-ideal.
		\item[(n)] If $ I $ is a (resp., strong) $ \H_{YJ} $-ideal, then $ \sqrt{I} $ is a (resp., strong) $ \H_{YJ} $-ideal.
		\item[(o)] Suppose that $ I $, $ J $ and $ K $ are ideals of $ R $. If $ IK $ is a (resp., strong) $ \H_{YJ} $-ideal, then $ I \cap K $ is a (resp., strong) $ \H_{YJ} $-ideal.
		\item[(p)] $ I $ is (resp., strong) $ \H_{YP} $-ideal, for some prime ideal $ P $ \ff (resp., $ I_{S\H} \setminus I $) $ I_\H \setminus I $ is a multiplicatively closed set.
		\item[(q)] If $ J $ is an ideals of $ R $ which not a (resp., strong) $ \H_Y $-ideal  and $ J \not \subseteq \bigcap Y $, then there is some strong $ \H_{YJ} $-ideal $ I \subset J $ which is not a (resp., strong) $ \H_Y $-ideal.
		\item[(r)] Suppose that $ I $ and $ J $ are ideals of $ R $ and $ P $ is a prime ideal of $ R $. If $ I \cap P $ is a (resp., strong) $ \H_{YJ} $-ideal, then either $ I $ or $ P $ is a (resp., strong) $ \H_{YJ} $-ideal.
		\item[(s)] Suppose that $ P $ and $ Q $ are prime ideals of $ R $ and $ J $ is an ideal of $ R $. If $ P \not\subseteq Q $, $ Q \not\subseteq P$ and $ P \cap Q $ is a (resp., strong) $ \H_{YJ} $-ideal, then both $ P $ and $ Q $ are (resp., strong) $ \H_{YJ} $-ideal.
		\end{itemize}
\label{property of relative J}
\end{Pro}
\begin{proof}
	(a) and (c). It is clear, according to the fact that the intersection of (resp., strong) $ \H_Y $-ideals is a (resp., strong) $ \H_Y $-ideal and Theorem \ref{Equivalnet condition by relative ideal}.
	
	(b), (d), (f), (h), (l). They follow easily from Theorem \ref{Equivalnet condition by relative ideal}.
	
	(e). \cite[Lemma 3.2]{badie2019extension}, \cite[Proposition 7.11]{badie2019extension} and Theorem \ref{Equivalnet condition by relative ideal}, conclude 
	\[ ( \sqrt{I} )_\H \cap \sqrt{J} = I_\H \cap \sqrt{J} = \sqrt{I_\H} \cap \sqrt{J} = \sqrt{I_\H \cap J} = \sqrt{I \cap J} = \sqrt{I} \cap \sqrt{J} \]
	Now Theorem \ref{Equivalnet condition by relative ideal}, follows that $ \sqrt{I} $ is an $ \H_{Y\sqrt{J}} $-ideal.
	
	(g). \cite[Lemma 2.4]{aliabad2013relative}, and Theorem \ref{Equivalnet condition by relative ideal}, conclude that $ I_\H \cap ( I + J ) = I_\H \cap I + I_\H \cap J \subseteq I $. Hence $ I $ is an $ H_{Y(I+J)} $-ideal, by Theorem \ref{Equivalnet condition by relative ideal}. The similar proof states for strong $ H_{YJ} $-ideals.
	
	(i $\Rightarrow$). Since $ I $ is a (resp., strong) $ \H_{YJ} $-ideal and $ J $ is a (resp., strong) $ \H_Y $-ideal, $ I \cap J $ is a (resp., strong) $ \H_{YJ} $-ideal. Now part (g) follows that $ I \cap J $ is a (resp., strong) $ \H_Y $-ideal.
	
	(i $\Leftarrow$). Since $ I \cap J $ is a (resp., strong) $ \H_{YJ} $-ideal, $ I $ is a (resp., strong) $ \H_{YJ} $-ideal, by part (f).
	
	(j). It follows immediately from (f).
	
	(k). It  is clear that $ I_\H \cap J $ is an $ \H_{YJ} $-ideal containing $ I \cap J $. Now suppose that $ K $ is an $ \H_{YJ} $-ideal containing $ I \cap J $, \cite[Proposition 5.1]{badie2019extension} and Theorem \ref{Equivalnet condition by relative ideal}, follow that
	\[ I_\H \cap J = I_\H \cap J_\H \cap J = ( I \cap J )_\H \cap J \subseteq K_\H \cap J \subseteq K  \]
	Similarly, we can prove it for strong $ \H_{YJ} $-ideals. 
	
	(m). Since $ I_\H \subseteq K_\H \subseteq $, $ I_\H = K_\H $. Now part (k) implies that $ K $ is an $ \H_{YJ} $-ideal. Similarly, we can prove it for strong $ \H_{YJ} $-ideals.
	
	(n). By \cite[Proposition 7.11]{badie2019extension}, $ I \subseteq \sqrt{I} \subseteq I_\H $, now part (m) completes the proof.
	
	(o). Proposition \cite[Proposition 5.1]{badie2019extension}, and Theorem \ref{Equivalnet condition by relative ideal}, deduce that
	\[ (I \cap K)_\H \cap J = ( I  K )_\H \cap J  \subseteq IK \subseteq I \cap K \]
	Thus $ I \cap K $ is an $ \H_{YJ} $-ideal. Similarly, we can prove it for strong $ \H_{YJ} $-ideals.
	
	(p $\Rightarrow$). Theorem \ref{Equivalnet condition by relative ideal}, follows that $ I_\H \cap P \subseteq I $. If $ a,b \in I_\H \setminus I $, then $ a,b \notin P $, so $ ab \notin P $, thus $ ab \notin P  $, since $ ab \in I_\H $, it follows that $ ab \in I_\H \setminus I $. Consequently $ I_\H \setminus I $ is a multiplicatively closed set.
	
	(p $\Leftarrow$). Since $ ( I_\H \setminus I ) \cap I = \emptyset $, there is some prime ideal $ P \in \Mi(I) $ such that $ ( I_\H \setminus I ) \cap P = \emptyset $, thus 
	\[ I_\H \cap P =  \big( ( I_\H \setminus I ) \cup I \big) \cap P \big) = I \cap P = I \]
	Now Theorem \ref{Equivalnet condition by relative ideal}, concludes that $ I $ is an $ \H_{YP} $-ideal.
	
	(q). Since $ J \not \subseteq \bigcap Y $, there is some $ P \in Y $ such that $ J \not \subseteq P $. Set $ I = P \cap J $, it is clear that $ I \subset J $ and $ I $ is a strong $ \H_{YJ} $-ideal. If $ I $ is a (resp., strong) $ \H_{YJ} $-ideal, then Proposition \ref{property of relative J}, follows that $ J $ is a (resp., strong), which contradicts our assumption.
	
	(r). It clear that if $ I \subseteq P $, then $ I $ is a $ \H_{YJ} $-ideal. Now suppose $ I \not\subseteq P $, so $ b \in I \setminus P $ exists. For each $ a  \in P $, $ ab \in I \cap J $, so $ kh_Y(a) \cap kh_Y(b) \cap J = kh_Y(ab) \cap J \subseteq I \cap P $, since  $ kh_Y(b) \not\subseteq P $, it follows that $ kh_Y (a) \cap J \subseteq P $. Consequently, $ P $ is $ \H_{YJ} $-ideal. The similar proof states for strong $ \H_{YJ} $-ideals.
	
	(s). By the assumption, $ P , Q \in \Mi(I \cap P) $, so $ P $ and $ Q $ are $ \H_{YJ} $-ideals, by Theorem \ref{Prime and H-ideal}(a).
\end{proof}

We say that a ring $R$ has the \emph{root property} if for each $ x \in R$ there are $ y \in R $ and $ 2 \leqslant n \in \N $ such that $y^n = x$.

\begin{Pro}
	Let $ Y \subseteq \Sp(R) $ and $ K = \bigcap Y $. The following hold.
	\begin{itemize}
	\item[(a)] An ideal $ I $  of $ R $ is a relative (resp., strong) $ \H_Y $-ideal \ff there is some ideal $ J \supset I $ such that $ I $ is a (resp., strong) $ \H_{YJ} $-ideal.
	\item[(b)] Suppose that $ I \subseteq K $ are ideals of $ R $ and $ I_\H = K_\H $ (resp., $ I_{S\H} = K_{S\H} $). If $ I $ is a relative (resp., strong) $ \H_Y $-ideal, then $ K $ is a relative (resp., strong) $ \H_Y $-ideal.
	\item[(c)] Suppose that $ I $ and $ K $ are ideals of $ R $. If $ I \subseteq K \subseteq I_\H $ (resp., $ I \subseteq K \subseteq K \subseteq I_{S\H} $) and $ I $ is a relative (resp., strong) $ \H_Y $-ideal, then $ K $ is a relative (resp., strong) $ \H_Y $-ideal.
	\item[(d)] If $ I $ is an ideal of $ R $ and $ P $ is a prime ideal of $ R $. If $ I \cap P $ is a relative (resp., strong) $ \H_Y $-ideal, then either $ I $ or $ P $ is a relative (resp., strong) $ \H_Y $-ideal.
	\item[(e)] An ideal $ I $ of a ring $ R $ is a relative (resp., strong) $ \H_Y $-ideal \ff there is some (resp., finite subset $ F \subseteq R $ which $ F \cap I = \emptyset $ ) $ c \in R \setminus I $ such that ($ \Ge{c} \cap kh_Y(F) \subseteq I $) $ \Ge{c} \cap kh_Y(a) \subseteq I $.
	\item[(f)] Suppose that $ I $ is an ideal of $ R $. If $ (K:I) \not\subseteq I $ and $ K \subseteq I $, then $ I $ is a relative strong $ \H_Y $-ideal.
	\item[(g)] Suppose that $ a \in R $, $ K \subseteq \Ge{a} $ and $ R $ has the root property. The principal ideal $ \Ge{a} $ is a relative strong $ \H_Y $-ideal \ff $ (K:a) \not\subseteq \Ge{a} $.
	\end{itemize} 
\label{property of realtive}
\end{Pro}
\begin{proof}
	(a). It is clear, by Proposition \ref{property of relative J}(g).
	
	(b). By the assumption, there is some ideal $ J‌ \not\subseteq I $ such that $ I $ is a (resp., strong) $ \H_{YJ} $-ideal, then Theorem \ref{property of relative J}(l), concludes that $ K $ is a (resp., strong) $ \H_{YJ} $-ideal. If $ J \subseteq K‌ \subseteq K_\H $, then $ J‌ \subseteq K_\H = I_\H \cap J \subseteq I $, which is a contradiction, so $ J‌ \not\subseteq K $ and therefore $ K $ is a relative $ \H_Y $-ideal.
	
	(c). It follows immediately from part (b).
	
	(d). It concludes immediately from Proposition \ref{property of relative J}(r).
	 	
	We prove  part (e) for $ \H_Y $-ideals, it is clear the similar proof states for strong $ \H_Y $-ideals.
	
	(e $\Rightarrow$). Since $ I $ is a relative $ \H_Y $-ideal, there is an ideal $ J \supseteq I $ such that $ I $ is an $ \H_{YJ} $-ideal. Thus $ c \in J \setminus I $ exists and $ I_H \cap J \subseteq I $, by Theorem \ref{Equivalnet condition by relative ideal}. For each $ a \in I \subseteq I_\H $ we have, $ kh_Y(a) \cap \Ge{c} \subseteq I_\H \cap J \subseteq I $.
	
	(e $\Leftarrow$). Set $ J = \Ge{c} $. Theorem \ref{Equivalnet condition by relative ideal} follows that $ I $ is a $ \H_{YJ} $-ideal, since $ J‌ \not\subseteq  I‌ $, it follows that $ I $ is a relative $ \H_Y $-ideal.
	
	(f). First we show that $ kh_Y (F) \cap (K:I) \subseteq K $, for each finite subset $ F $ of $ R $. Suppose that $ F $ is a finite subset of $ I $ and $ a \in kh_Y(F) \cap (K:I) $, then $ kh_Y(a) \subseteq kh_Y(F) $ and $ aI \subseteq K $, thus $ kh_Y(a) = kh_Y(aF) $ and $ aF $ is a finite subset of $ K $, since $ K $ is a strong $ \H_Y $-ideal, it follows that $ a \in kh_Y(aF) \subseteq K $. Consequently, $ kh_Y (F) \cap (K:I) \subseteq K $. Now take $ J = (K:I) $, then $ c \in J \setminus I $ exists. For each finite subset $ F $ of $ I $, we have 
	\[ \Ge{c} \cap kh_Y(F) \subseteq J \cap kh_Y(F) = (K:I) \cap kh_Y(F) \subseteq K \subseteq I \]
	Now part (c) follows that $ I $ is a relative strong $ \H_Y $-ideal.
	
	(g $ \Rightarrow $). Since $ \Ge{a} $ is a relative strong $ \H_Y $-ideal, there is some ideal $ J \not\subseteq \Ge{a} $ such that $ \Ge{a} $ is a strong $ \H_{YJ} $-ideal, so $ kh_Y(a) \cap J \subseteq \Ge{a} $ and $ c \in J \setminus \Ge{a} $ exists, thus $ c \in J \setminus kh_Y(a) $. Since $ R $ has the root property, there is some $ b \in R $ and $ 2 \leqslant n \in \N $ such that $ a = b^n $. Then $ bc \in J \cap kh_Y(a) \subseteq \Ge{a} $, so $ bc = ar = b^n r $, thus $ b ( c - n^{n-1} r ) = 0 $. We claim $ c - b^{n-1} \notin \Ge{a} $, on contrary, $ c \in \Ge{a} \subseteq kh_Y(a) $, which is a contradiction. Hence $ c - b^{n-1} \in (K:b) = (K:a) $ and therefore $ c \in (K:a) \setminus \Ge{a} $.
	
	(g $ \Leftarrow $).  It follows from part (f).
\end{proof}

A ring $ R $ is called \emph{arithmetical}, if for every ideal $ I $, $ J $ and $ K $ of $ R $, we have
\[ I \cap ( J + K ) = ( I \cap J ) + ( I \cap K) \]

\begin{Pro}
	Suppose that $ Y \subseteq \Sp(R) $ and $ J $ is an ideal of $ R $.
	\begin{itemize}
		\item[(a)] $ I $ is a maximal element of  $ \{ I : J \text{ is a (resp., strong) } \H_{J}\text{-factor of } I  \} $ \ff $ I $ is a maximal element of 
		\[ \{ P : P \text{ is a prime (resp., strong) } \H_Y\text{-ideal and } J \not\subseteq P \}. \]
		\item[(b)] Every maximal element of $ \{ I \subset J : I \text{ is a (resp., strong) } \H_{YJ} \text{-ideal } \} $ is of the form $ P \cap I $, in which $ P $ is a maximal element of
		\[ \{ P : P \text{ is a prime (resp., strong) } \H_Y\text{-ideal and } J \not\subseteq P \}. \]
		\item[(c)] Suppose that $ I $ is an ideal of $ R $. If $ I $ has a (resp., strong) $ \H_Y $-factor, then the family of all (resp., strong) $ H_Y $-factors of $ I $ has a maximal element, which contains $ I $.
		\item[(d)] If $ J $ is a minimal (resp., strong) $ \H_Y $-factor of an ideal $ I $, then $ I + J $ is a minimal (resp., strong) $ \H_Y $-factor of $ I $ containing $ I $.
		\item[(e)] If the largest (resp., strong) $ \H_Y $-factor of an ideal $ I $ exists, then it is of the form (resp., $ \{ x \in R : kh_Y(x) \cap \Ge{F} \subseteq I \text{ for all finite subset } F \text{ of } I \} $) $ \{ x \in R : kh_Y(x) \cap \Ge{a} \subseteq I \text{ for all  } a \in I \} $.
		\item[(f)] Suppose $ I $ is an ideal of $ R $. If 
		\[ K = \bigcap \{ P \in \Mi(I) : I \text{ is not a (resp., strong) } \H_Y \text{-ideal } \}, \]
		then (resp.,  $ I_{S\H} \cap K \subseteq \sqrt{I} $) $ I_\H \cap K \subseteq \sqrt{I} $. Also if $ I $ is a relative semi-prime ideal, then $ K $ is the greatest (resp., strong) $ \H_Y $-factor of $ I $.
		\item[(g)] If $ I $ is a relative (resp., strong) $ \H_Y $-ideal, then the greatest (resp., strong)‌ $ \H_Y $-factor of $ I $ exists.
		\item[(h)] If $ R $ is arithmetical, then every relative (resp., strong)‌ $ \H_Y $-ideal has the greatest (resp., strong) $ \H_Y $-factor.
		\item[(i)] Suppose $K = \bigcap Y $.  ideal $ J $ is a minimal (resp., strong) $ \H_Y $-factor of an ideal $ I $ \ff $ J = \Ge{e} $, for some $ e \notin \sqrt{I} $, such that $ I  \cap \Ge{K,e} $ is a (resp., strong) $ \H_Y $-ideal and $ \Ge{e} $ is a minimal principal ideal which is not contained in $ I $.   
	\end{itemize}
\label{H_Y factor}
\end{Pro}
\begin{proof}
	(a). We show this part for $ \H_Y $-ideals, the similar proof states for strong $ \H_{YJ} $-ideals. Suppose that $ I \in \maxl \{ I : J \text{ is a } \H_{Y}\text{-factor of } I \} $. Since $ J $ is an $ \H_Y $-factor of every prime $ \H_Y $-ideal which not contains $ J $, it is sufficient to $ I $ is a prime $ \H_Y $-ideal. We claim $ J \not \subseteq I_\H $, otherwise, Theorem \ref{Equivalnet condition by relative ideal}, concludes $ J = I_\H \cap J \subseteq I $, which is a contradiction. Since $ I_\H $ is an $ \H_{YJ} $-ideal, it follows that $ I_\H = I $, by the maximality of $ I $, hence $ I $ is an $ \H_Y $-ideal, consequently, $ I $ is semi-prime, by \cite[Lemma 3.2]{badie2019extension}. Since $ I \not \subseteq J $, there is some $ P \in \Mi(I) $ such that $ J \not \subseteq P $. $ P $ is an $ \H_{YJ} $-ideal, by Theorem \ref{Prime and H-ideal}. Hence $ P = I $, by the maximality of $ I $.
	
	(b). Suppose that $ I $ a maximal element of $ \{ I \subset J : I \text{ is a } \H_{YJ} \text{-ideal } \} $. If $ J \subseteq I_\H $, then Theorem \ref{Equivalnet condition by relative ideal}, concludes $ J = J \cap I_\H \subseteq I  $, which is a contradiction. Hence $ J \not\subseteq I_\H $, so there is some $ Q \in \Mi(I_\H) $, such that $ J \not\subseteq P $. By \cite[Theorem 3.13]{badie2019extension}, $ Q $ is a prime $ \H_Y $-ideal. So there is some maximal element $ P $ of $ \{ P : P \text{ is a prime } \H_Y\text{-ideal and } J \not\subseteq P \} $ such that contains $ Q $. Thus $ I = I \cap J \subseteq Q \cap J \subseteq P \cap J \subseteq J $, since $ P \cap J $ is an $ \H_{YJ} $-ideal, it follows that $ I = P \cap J $.
	
	(c). If $ \mathscr{C} $ is a chain of $ H_Y $-factor of $ I $, then $ \big( \bigcup_{J\in \mathscr{C} } J \big) \cap I_\H = \bigcup_{J\in \mathscr{C} } ( J \cap I_\H ) \subseteq I  $ and clearly, $ \bigcup_{J\in \mathscr{C} } J \not\subseteq I $, so $ \bigcup_{J\in \mathscr{C} } J $ is an $ \H_Y $-factor of $ I $, thus the family of all $ \H_Y $-factor of $ I $ has a maximal element, by Zorn's lemma. Now Proposition \ref{property of realtive}(a), follows that $ I \subset J $.
	
	(d). Since $ J $ is a (resp., strong) $ \H_Y $-factor of $ I $, $ I + J $ is a $ \H_Y $-factor of $ I $, by Proposition \ref{property of relative J}(g). If $ K \subseteq I + J  $ is an $ \H_Y $-factor of $ I $ containing $ I $, then $ k \in K \setminus I $ exists, hence there are $ i \in I $ and $ j \in J $ such that $ k = i + j $, so $ j = k - i \in K \setminus I $, and therefore $ j \in ( K \cap J ) \setminus I $, so $ K \cap J $ is an $ \H_Y $-factor of $ I $, by Proposition \ref{property of relative J}(a). Now the minimality of $ J $ concludes that $ K \cap J = J $, hence $ J \subseteq K $, and therefore $ I + J \subseteq K $.
	
	(e). Suppose $ J $ is the largest $ \H_Y $-factor of $ I $ and $ K = \{ x \in R : kh_Y(x) \cap \Ge{a} \subseteq I \quad \forall a \in I \} $. $ x , y \in K $ implies that $ kh_Y(a) \cap \Ge{x} \subseteq I $ and $ kh_Y(a) \cap \Ge{y} \subseteq I $, for all $ a \in I $, so $ \Ge{x}  $ and $ \Ge{y} $ are $ \H_Y $-factors of $ I $, since $ J $ is the greatest $ \H_Y $-factor of $ I $, it follows that $ x , y \in J $, thus $ kh_Y(a) \cap \Ge{x+y} \subseteq kh_Y(a) \cap J \subseteq I $, for all $ a \in I $, hence $ x + y \in K $. Now suppose $ x \in K $ and $ r \in R $, then $ kh_Y(a) \cap \Ge{rx} \subseteq kh_Y(a) \cap \Ge{x} \subseteq I $ and therefore $ rx \in K $. Consequently, $ K $ is an ideal of $ I $. Also for every $ a \in I $,
	\[ kh_Y(a) \cap K = kh_Y(a) \cap \big(\bigcup_{x \in K} \Ge{x} \big) = \bigcup_{x \in K}  kh_Y(a) \cap \Ge{x} \subseteq I \]
	Hence $ I $ is an $ \H_{YK} $-ideal. Since for each $x \in J$ we have  $ kh_Y(a) \cap \Ge{x} \subseteq kh_Y(a) \cap J \subseteq I $, for all $ a \in I $, it follows that $ J \subseteq K $, and thus $ J = K $. Similarly, we can show this part for strong $ \H_Y $-factors.
	
	(f). Set $ \cal{A} = \{ P \in \Mi(I) : P \text{ is an } \H_Y \text{-ideal} \} $ and $ \cal{B} = \Mi(I) \setminus \cal{A}  $.  By Proposition \cite[Proposition 7.11]{badie2019extension},
	\begin{align*}
	I_\H \cap K & = ( \sqrt{I} )_\H \cap K  =  \big( \bigcap_{ P \in \Mi(I) } P \big)_\H \cap K \\
				& = \big( \bigcap_{ P \in \Mi(I) } P_\H \big) \cap K = \big( \bigcap_{ P \in \cal{A} } P \big) \cap \big( \bigcap_{ P \in \cal{B} } P_H \big) \cap \big( \bigcap_{ P \in \cal{B} } P \big) \\
				& = \big( \bigcap_{ P \in \cal{A} } P \big) \cap \big( \bigcap_{ P \in \cal{B} } P \big) = \sqrt{I}
	\end{align*}
	Now suppose that $ I $ is a semi-prime ideal, then $ I_H \cap K = \sqrt{I} = I $, and therefore $ I $ is an $ \H_{YK} $-ideal. If $ J $ is an $ \H_{Y} $-factor of $ I $, then for each $ P \in \cal{B} $, then every $ P \in \cal{B} $ is an $ \H_{YJ} $-ideal, by Theorem \ref{Prime and H-ideal}(a), since every $ P \in \cal{B} $ is a not $ H_Y $-ideal, it follows that $ J \subseteq P $, by Theorem \ref{Prime and H-ideal}(b), hence $ J \subseteq K $. Since $ I $ is a relative $ \H_Y $-ideal, it has a $ \H_Y $-factor $ J \subseteq K $ and therefore $ K \not\subseteq I $. Consequently, $ K $ is the greatest $ \H_Y $-factor of $ I $.
	
	(g). Since $ I $ is a relative (resp., strong) $ \H_Y $ -ideal, $ \sqrt{I} $ is $ \H_Y $-ideal, by Proposition \ref{property of realtive}(b). Now part (f) deduces that $ \sqrt{I} $ has the greatest  (resp., strong) $ \H_Y $-factor.
	
	(h). Suppose that $ I $ is a relative (resp., strong)‌ $ \H_Y $-ideal. Part (c) deduces that $ I $ has a maximal (resp., strong) $ \H_Y $-factor $ J $. If $ K $ is a $ \H_Y $-factor of $ I $, then 
	\[ I_\H \cap (J+K)‌ = I_\H \cap J +‌ I_\H \cap J \subseteq I \]
	Thus $ J +‌ K $  is a $ \H_Y $-factor of $ I $, and therefore $ K \subseteq J $, by the maximality of $ J $. Hence $ J $ is the greatest $ \H_Y $-factor of $ I $.
	
	We prove part (i) for $ \H_Y $-ideals, similarly one can show it for strong $ \H_Y$-ideals
	
	(i$\Rightarrow$). Set $ e \in J \setminus I $. Proposition \ref{property of relative J}(b), concludes that $ I $ is an $ \H_{Y\Ge{e}} $-ideal, hence $ J = \Ge{e} $, by the minimality of $ J $. If for some $ n \in \N $, $ e^n \in I $, then $ e \in kh_Y(e) \cap J = kh_Y(e^n) \cap J \subseteq I $, which is a contradiction. Hence $ e \notin \sqrt{I} $. Since $ \Ge{e^2} \subseteq \Ge{e} $, $ I $ is an $ \H_{Y\Ge{e^2}} $-ideal, by Proposition \ref{property of relative J}(b). The minimality of $ J $ follows that $ \Ge{e} = \Ge{e^2} $, thus  for some $ r \in R $, $ e = e^2 r $. Now we show that $ \Ge{e,K} = \bigcap_{ 1-er \notin P \in Y } P $. If $ 1 - er \notin P $, then $ e (1 - er ) = e - e^2r = 0 \in P $, so $ e \in R $, hence $ \Ge{e,K} \subseteq P $ and therefore $ \Ge{e,K} \subseteq \bigcap_{ 1-er \notin P \in Y } P $. If $ y \in \bigcap_{ 1-er \notin P \in Y } P $, since $ \bigcap_{ 1-er \notin P \in Y } P = ( K : 1-er ) $, it follows $ y -yer = y (1-er) \in K $ and thus $ y \in \Ge{K,e} $. Consequently, $ \Ge{K,e} = \bigcap_{ 1-er \notin P \in Y } P  $. Since $ \bigcap_{1-er \notin P \in Y  } P $ is a $ Y $-Hilbert, $ \Ge{e} $ is an $ \H_Y $-ideal. Proposition \ref{property of relative J}(i),  deduces that $ I \cap \Ge{K,e} $ is an $ \H_Y $-ideal. If $ \Ge{a} \subseteq \Ge{e} $, then $ I $ is an $ \H_{Y\Ge{a}} $-ideal, by Proposition \ref{property of relative J}(b). Now the minimality of $ J $ implies that $ \Ge{a} \subseteq I $, hence $ \Ge{e} $ is a minimal principal ideal which not contained in $ I $.
	
	(i$\Leftarrow$). Since $ J = \Ge{e} $ and $ \Ge{K,e} \cap I $ are $ \H_Y $-ideals, $I$ is an $ \H_{Y\Ge{K,e}} $-ideal, by Proposition \ref{property of relative J}(i). Proposition \ref{property of relative J}(b), implies that $ I $ is a $H_{YJ} $-ideal, since $ e \notin I $, it follows that $ J $ is an $ \H_Y $-factor of $I$. If $ K \subset \Ge{e} $ is an $ \H_Y$-factor of $ I $, then $ k \in K \setminus I $ exists, then $ \Ge{k} \subseteq K \subset \Ge{e} $ and $ \Ge{k} \not\subseteq I $, which contradicts the minimality of $ \Ge{e} $. Hence $ J $ is a minimal $\H_Y$-factor of $I$.
\end{proof}

\begin{Cor}
	Suppose $ Y \subseteq \Sp(R) $ and $ I $ is a semi-prime ideal of $ R $. $ I $ is a relative (resp., strong) $ \H_Y $-ideal \ff in each representation of $ I $ as the intersection of prime ideals, there is a (resp., strong) $ \H_Y $-ideal.
\end{Cor}
\begin{proof}
	($ \Rightarrow $). By Proposition \ref{H_Y factor}(f), $ I $ has the greatest $ \H_Y $-factor $ J $. Proposition \ref{property of relative J}(e), concludes $ J $ is a semi-prime. If $ I = \bigcap_{ P \in \cal{A} } $ is a representation of prime ideals, then Theorem \ref{Prime and H-ideal}(a), deduces that $ P_\alpha $ is an $ \H_{YJ} $-ideal, for each $ \alpha \in A $. If for every $ \alpha \in \cal{A} $, $ P_\alpha $ is not an $ \H_Y $-ideal, then Theorem \ref{Prime and H-ideal}(b), concludes $ P_\alpha \supseteq J $ and therefore $ I = \bigcap_{ P \in \cal{A} } P \supseteq J $, which is a contradiction.
	
	($ \Leftarrow $). Set
	\[ J = \bigcap \{ P \in \Mi(I) : I \text{ is not a (resp., strong) } \H_Y \text{-ideal } \}. \]
	Then $ J \not\subseteq I $ and therefore Theorem \ref{H_Y factor}(f), concludes $ I $ is a relative $ \H_Y $-ideal.
\end{proof}

\begin{Thm}
	Suppose that $X,Y \subseteq \Sp(R)$ and $J$ is an ideal of $R$. Every (resp., strong) $\H_{XJ}$-ideal is a (resp., strong) $\H_{YJ}$-ideal \ff every prime (resp., strong) $\H_X$-ideal which is not containing $J$ is a (resp., strong) $\H_Y$-ideal.
\end{Thm}
\begin{proof}
	$\Rightarrow$). Suppose that $P$ is a prime (resp., strong) $\H_X$-ideal which is not containing $J$, then $P$ is a (resp., strong) $\H_{XJ}$-ideal and therefore $P$ is a (resp., strong) $\H_{YJ}$-ideal, by the hypothesis. Since $J \not\subseteq P$, $P$ is $\H_Y$-ideal, by Theorem \ref{Prime and H-ideal}. 
	
	$\Leftarrow$). Set $Z$ the family of all prime (resp., strong) $\H_X$-ideals which are not containing $J$. Suppose that $I$ is a (resp., strong) $\H_{XJ}$-ideal. We claim $I$ is a (resp., strong) $\H_{ZJ}$-ideal, on contrary $a \in I$ and $b \in J$  exist such that $h_Z(a) \subseteq h_Z(b)$ and $ b \notin I $, then $ P \in \Mi(I) $ exists such that $ b \notin P$, so $ J \not\subseteq P $, hence $P$ is a (resp., strong) $\H_X$-ideal, by Theorem \ref{Prime and H-ideal}. Thus $ P \in Z $, since $ h_Z(a) \subseteq h_Z(b) $ and $ a \in I \subseteq P $, it follows that $ b \in kh_Z(b) \subseteq kh_Z(a) \subseteq P $, which is a contradiction. Hence $I$ is a (resp., strong) $\H_{ZJ}$-ideal, so there is some (resp., strong) $\H_Z$-ideal $K$ such that $ K \cap J \subseteq I $, by Theorem \ref{Equivalnet condition by relative ideal}. Since every element of $Z$ is a (resp., strong) $\H_Y$-ideal, $K$ is a (resp., strong) $\H_Y$-ideal, by \cite[Corollary 6.4]{badie2019extension}. Now Theorem \ref{Equivalnet condition by relative ideal}, concludes $I$ is a (resp., strong) $\H_{YJ}$-ideal.
\end{proof}

\begin{Thm}
	Suppose $ Y \subseteq \Sp(R) $ and $ \bigcap Y = \{ 0 \} $. The following statements are equivalent.
	\begin{itemize}
		\item[(a)] Every proper ideal of $ R $ is a relative strong $ \H_Y $-ideal.
		\item[(b)] Every proper ideal of $ R $ is a relative $ \H_Y $-ideal.
		\item[(c)] $ R $ is a regular ring.
	\end{itemize}
\end{Thm}
\begin{proof}
	(a $ \Rightarrow $ b). It is clear.
	
	(b $ \Rightarrow $ c). Suppose that $ I $ is a proper ideal of $ R $, then $ I $ is a relative $ \H_Y $-ideal, so it has a $ \H_Y $-factor and therefore it has a maximal $ \H_Y $-factor $ J $, by Proposition \ref{H_Y factor}. If $ J \neq R $, then $ J $ is a $ H_Y $-ideal, thus has a $ \H_Y $-factor $ K $ containing $ J $, by  Proposition \ref{property of realtive}. Now Proposition \ref{property of relative J}, concludes that $ K $ is a $ \H_Y $-factor of $ I $, which contradicts the maximality of $ J $. Hence $ J = R $, and therefore $ I $ is a $ \H_Y $-ideal. Thus $ R $ is a regular ring, by  \cite[Proposition 5.8]{badie2019extension}.
	
	(c $ \Rightarrow $ a). Suppose that $ I $ is a proper ideal of $ R $, then \cite[Proposition 5.8]{badie2019extension}, concludes that $ I $ is a strong $ \H_Y $-ideal and therefore $ I $ is a strong $ \H_Y $-ideal, by Proposition \ref{Primary relative}.
\end{proof}


\begin{thebibliography}{20}
	\bibitem{aliabad2013relative} Aliabad, A.R., Azarpanah, F., and Taherifar, A., Relative $z$-ideals in commutative rings. Comm. Algebra 41, 1 (2013), 325–341.
	
	\bibitem{aliabad2013fixed} Aliabad, A.R. and Badie, M., Fixed-place ideals in commutative rings. Comment. Math. Univ. Carolin. 54, 1 (2013), 53–68.
	
	\bibitem{badie2019extension} Aliabad, A.R., Badie, M. and Nazari, S., An extension of $z$-ideals and $z^\circ$-ideals. Hacet. J. Math. Stat. to appear
	
	\bibitem{aliabad2011sz} Aliabad, A.R. and Mohamadian, R., On $sz^\circ$-ideals in polynomial rings. Comm. Algebra 39, 2 (2011), 701–717.
	
	\bibitem{Artico1981} Artico, G., Marconi, U., and Moresco, R., A subspace of $Spec(A)$ and its connexions with the maximal ring of quotients. Rend. Semin. Mat. Univ. Padova 64 (1981), 93–107.
	
	\bibitem{atiyah1969introduction} Atiyah, M. F., and Macdonald, I. G., Introduction to commutative algebra, vol. 2. Addison-Wesley Reading, MA, 1969.
	
	\bibitem{azarpanah1999z} Azarpanah, F., Karamzadeh, O., and Rezai, A. Aliabad, On $z^\circ$-ideals in $C(X)$. Fund. Math. 160, 1 (1999), 15–25.
	
	\bibitem{azarpanah2000ideals} Azarpanah, F., Karamzadeh, O., and Rezai, A. Aliabad, On ideals consisting entirely of zero divisors. Comm. Algebra 28, 2 (2000), 1061–1073.
	
	\bibitem{gillman1960rings} Gillman, L., and Jerison, M., Rings of continuous functions. Van. Nostrand Reinhold, New York, 1960.
	
	\bibitem{huijsmans1980z} Huijsmans, C., and De Pagter, B., On $z$-ideals and $d$-ideals in riesz spaces. i. Ind. Math. (Proceedings) 83, 2 (1980), 183–195.
	\bibitem{lam1991first} Lam, T., A first course in noncommutative rings. Springer, 1991.
	
	\bibitem{mason1973z} Mason, G., $z$-ideals and prime ideals. J. Algebra 26, 2 (1973), 280–297.
	
	\bibitem{stephen1970general} Willard, S., General Topology. Addison Wesley Publishing Company, New York, 1970.
\end{thebibliography}
\end{document}